\documentclass[reqno]{amsart}
\usepackage[margin=1in]{geometry}
\usepackage{xcolor}
\usepackage{enumitem}

\definecolor{bluecite}{HTML}{0875b7}
\usepackage[unicode=true,
bookmarksopen={true},
pdffitwindow=true,
colorlinks=true,
linkcolor=bluecite,
citecolor=bluecite,
urlcolor=bluecite,
hyperfootnotes=false,
pdfstartview={FitH},
pdfpagemode= UseNone]{hyperref}

\newcommand{\diff}{\,\mathrm{d}}
\DeclareMathOperator{\di}{div}
\DeclareMathOperator{\he}{Hess}

\newtheorem{theorem}{Theorem}[section]
\newtheorem{proposition}{Proposition}[section]
\newtheorem{corollary}{Corollary}[section]
\theoremstyle{definition}
\newtheorem{definition}{Definition}[section]
\newtheorem{remark}{Remark}[section]
\newtheorem{example}{Example}[section]
\numberwithin{equation}{section}

\title[Rellich inequalities via Riccati pairs]{Rellich inequalities via Riccati pairs on model space forms}
\author{S\'andor Kaj\'ant\'o}\address{ \textsc{S\'andor Kaj\'ant\'o}: Department of Mathematics and Computer Science, Babe\c s-Bolyai University, Cluj-Napoca, Romania}
\email{sandor.kajanto@ubbcluj.ro}
\date{\today}


\thanks{The author was supported by the UEFISCDI/CNCS grant PN-III-P4-ID-PCE2020-1001.}

\begin{document}
    \begin{abstract} 
       We present a simple method for proving Rellich inequalities on Riemannian manifolds with constant, non-positive sectional curvature. The method is built upon simple convexity arguments, integration by parts, and the so-called Riccati pairs, which are based on the solvability of a Riccati-type ordinary differential inequality. These results can be viewed as the higher order counterparts of the recent work by Kajántó, Kristály, Peter, and Zhao, discussing Hardy inequalities using Riccati pairs.
    \end{abstract}

    \maketitle
    \vspace{-0.5cm}
	\tableofcontents
    \vspace{-1cm}

    \section{Introduction} \label{sec:intro}
    The Hardy-type inequalities are functional inequalities involving the integrals of a given smooth function and its gradient with some -- possibly non-singular -- potentials. Formally, they can be written as
    \begin{equation}\label{eq:intro:Hardy}
        \int_\Omega P_1 |\nabla u|^2\diff {\sf m} \ge \int_\Omega Q_1 u^2\diff {\sf m},\quad \forall u\in C_0^\infty(\Omega),
    \end{equation}
    where \(\Omega\) is an open subset of the ambient space \(M\) (which could be the Euclidean space, any Riemann/Finsler manifold or stratified group), \(\sf m\) is a measure on \(M\) and \(P_1,Q_1\colon\Omega\to(0,\infty)\) are some given potentials. These inequalities originate from the celebrated work of Hardy~\cite{hardy1920note} from the 1920s, where the author proved inequality~\eqref{eq:intro:Hardy} and its \emph{sharpness} for the particular choices  \(\Omega=M=\mathbb{R}^n\), \(P_1\equiv1\) and \(Q_1(x)=\frac{(n-2)^2}{4|x|^2}.\) For detailed discussions and additional results, we refer to the monographs by Balinsky, Evans, and Lewis \cite{balinsky2015analysis}, Ghoussoub and Moradifam \cite{ghoussoub2013functional}, and Ruzhansky and Suragan \cite{ruzhansky2019hardy}.

    The higher order variants of the inequality~\eqref{eq:intro:Hardy} are the following Rellich-type inequalities:
    \begin{align}
        \label{eq:intro:Rellich:gru}\int_\Omega P_2 |\Delta u|^2\diff {\sf m} &\ge \int_\Omega Q_2 |\nabla u|^2\diff {\sf m},\quad \forall u\in C_0^\infty(\Omega),\\
        \label{eq:intro:Rellich:u}\int_\Omega P_3 |\Delta u|^2\diff {\sf m} &\ge \int_\Omega Q_3 u^2\diff {\sf m},\quad \forall u\in C_0^\infty(\Omega).
    \end{align}
    The sharp Euclidean version of~\eqref{eq:intro:Rellich:u} with \(V_3\equiv 1\) and \(Q_3(x)=\frac{n^2(n-4)^2}{16|x|^4}\) dates back to the 1950s and is due to Rellich~\cite{rellich1954halbbeschrankte}. Surprisingly, as claimed by the authors, the sharp Euclidean version of~\eqref{eq:intro:Rellich:gru} with \(P_2\equiv 1\) and \(Q_2(x)= \frac{n^2}{4|x|^2}\) only appeared relatively recently (in the 2000s) in the paper by Tertikas and Zographopoulos~\cite{tertikas2007best}. These facts suggest that the inequality~\eqref{eq:intro:Rellich:gru} is more problematic. Pioneering results in the Euclidean setting are due to  Davies and Hinz~\cite{davies1998explicit} and Mitidieri~\cite{mitidieri1993rellich}. For results on more general structures see Kombe and \"Ozaydin~\cite{kombe2009improved,kombe2013hardy} and Kristály and Repov\v{s}~\cite{kristaly2016quantitative}.

    As we shall see, the inequality~\eqref{eq:intro:Rellich:u} follows from the inequalities~\eqref{eq:intro:Hardy} and~\eqref{eq:intro:Rellich:gru} for the choices \(P_3=P_2\), \(Q_2=P_1\) and \(Q_3=Q_1\). It is also worth to point out, that there exist alternative versions of the latter inequalities involving the radial derivatives of the unknown function:
    \begin{equation}\label{eq:intro:radial}
        \int_\Omega P_1 |\nabla^{\rm rad} u|^2\diff {\sf m} \ge \int_\Omega Q_1 u^2\diff {\sf m}\quad\mbox{and}\quad\int_\Omega P_2 |\Delta u|^2\diff {\sf m} \ge \int_\Omega Q_2 |\nabla^{\rm rad} u|^2\diff {\sf m},
    \end{equation}
    where \(\nabla^{\rm rad}u=\langle \nabla u,\nabla d_{x_0}\rangle\) and \(d_{x_0}\) is the distance from a fixed point \(x_0\in\Omega\). Combining these inequalities also implies a valid proof for inequality~\eqref{eq:intro:Rellich:u} as well; see e.g.~Nguyen~\cite{nguyen2020new}. 

    In a celebrated paper by Ghoussoub and Moradifam~\cite{ghoussoub2011bessel}, the authors characterized inequality~\eqref{eq:intro:Hardy}, and gave sufficient condition for inequality~\eqref{eq:intro:Rellich:gru} to hold in the Euclidean setting. Arguments are built upon the spherical harmonics decomposition and the notion of \emph{Bessel pairs}, which are based on the solvability of a second-order Bessel-type ordinary differential equation. Despite the fact that this notion is typically Euclidean (encodes no curvature information), several authors managed to apply it in hyperbolic spaces (Riemannian manifolds with constant negative sectional curvature) using comparisons for Jacobi fields; see e.g.~Flynn, Lam, Lu, and Mazumdar~\cite{flynn2023hardy} and Berchio, Ganguly, and Grillo~\cite{berchio2017sharp}. 
    
    In a recent paper by Kajántó, Kristály, Peter, and Zhao~\cite{riccatipair2023}, the authors presented a simple alternative approach to prove Hardy-type inequalities on general complete, non-compact Riemannian manifolds with sectional curvature \({\bf K}\le-\kappa^2\), for some \(\kappa\le0\). Their method is built upon convexity arguments, integration by parts, Laplace comparison, and the notion of \emph{Riccati pairs}, which is based on the solvability of a first-order Riccati-type ordinary differential inequality. It is worth to point out that in the Euclidean setting, the notion of Bessel pairs and Riccati pairs are `essentially' equivalent (see Proposition~\ref{prop:eu:relations}), moreover, as was shown in~\cite[Proposition 3.1]{riccatipair2023}, there exists a natural extension of Bessel pairs on curved spaces, which preserves this equivalence. However, Riccati pairs \emph{naturally appear} from our convexity arguments; hence, they are easier to use to present our results. We emphasize that the crucial difference is between the argument laying behind these concepts. For an alternative, convexity based approach see Farkas, Kajántó and Kristály~\cite{farkas2023sharp}.

    The goal of this paper is to investigate the capabilities of convexity arguments and Riccati pairs in the context of Rellich inequalities on non-positively curved Riemannian manifolds. Our first important task is to choose a suitable ambient space. The convexity arguments corresponding to Rellich-type inequalities require a Hessian comparison as well, which has `opposite' direction to the Laplace comparison. Thus, one needs to either impose a lower bound on the sectional curvature as well, or simply consider manifolds with constant sectional curvature, also called \emph{model space forms}. We decided on the second option because the relevant applications are also formulated in this setting. Additionally, in this particular case, the results can be presented in a formally simpler and more accessible manner.

    We give a sufficient condition for the inequality~\eqref{eq:intro:Rellich:gru} to hold, and hence its combination with~\eqref{eq:intro:Hardy} implies~\eqref{eq:intro:Rellich:u}. It turns out, however, that this condition typically implies a dimension constraint, meaning that the method only works in high dimensions. To overcome this problem we exploit the idea of radial derivatives: we give sufficient conditions for inequalities in~\eqref{eq:intro:radial} and combine them to prove~\eqref{eq:intro:Rellich:u} without dimension constraint. For the sake of illustration, let us present the following simplified Euclidean version of our first main result (for the general variant, see Theorem~\ref{thm:main1}).
    \begin{theorem}\label{thm:simp}
        Let \(\Omega\subset\mathbb{R}^n\) be an open domain, \(n\ge 5\), \(V\colon(0,\infty)\to [0,\infty)\) and \(H\colon(0,\infty)\to\mathbb{R}\) such that the following ordinary differential inequality holds:
        \begin{equation}\label{eq:Riccati:odi:simp}
            -H'(t)+H(t)\cdot \frac{(n-1)}{t}-H(t)^2\ge V(t),\qquad \forall t>0.
        \end{equation}
        \begin{enumerate}[label=\rm(\roman*)]
            \item\label{c:is} If \(E_1(t)=H'(t)+H(t)\cdot\frac{n-3}{t}\ge 0\) for all \(t>0\), then for every \(u\in C_0^\infty(\Omega)\) one has 
            \[
                \int_\Omega |\Delta u|^2\diff x\ge \int_\Omega V(|x|)|\nabla^{\rm rad} u|^2\diff x.
            \]
            \item\label{c:iis} If \(E_2(t)=2H'(t)+H(t)(H(t)-\frac{2}{t})\ge 0\) for all \(t>0\), then for  every \(u\in C_0^\infty(\Omega)\) one has 
            \[
                \int_\Omega|\Delta u|^2\diff x\ge \int_\Omega V(|x|)|\nabla u|^2\diff x.
            \]
        \end{enumerate}    
    \end{theorem}

    Theorem~\ref{thm:simp}/\ref{c:is} can be used efficiently in combination with our second main result Theorem~\ref{thm:main2}, which makes a connection between the weighted integrals of \(|\nabla^{\rm rad} u|^2\) and \(u^2\). To demonstrate this, we prove the classical Rellich-type inequalities as follows. 
    
    Let us choose \(H(t)=\frac{n}{2t}\) and \(V(t)=\frac{n^2}{4t^2}\) in Theorem~\ref{thm:simp}. Observe that~\eqref{eq:Riccati:odi:simp} is verified with equality, moreover, \(E_1(t)=\frac{n(n-4)}{2t^2}\) and  \(E_2(t)=\frac{n(n-8)}{4t^2}\), hence for very
    \(u\in C_0^\infty(\Omega)\) one has
    \begin{align*}
        \int_\Omega |\Delta u|^2\diff x&\ge \frac{n^2}{4}\int_\Omega\frac{|\nabla^{\rm rad} u|^2}{|x|^2}\diff x,\qquad\mbox{(if \(n\ge 5\))},\\
        \int_\Omega |\Delta u|^2\diff x&\ge \frac{n^2}{4}\int_\Omega\frac{|\nabla u|^2}{|x|^2}\diff x,\qquad\mbox{(if \(n\ge 8\))}.
    \end{align*}
    The first inequality and Theorem~\ref{thm:main2} (with choices \(w(t)=\frac{n^2}{4t^2}\), \(G(t)=\frac{n-4}{2t}\) and \(W(t)=\frac{n^2(n-4)^2}{16t^4}\)) imply for very \(u\in C_0^\infty(\Omega)\) that
    \[
        \int_\Omega |\Delta u|^2\diff x\ge\frac{n^2}{4}\int_\Omega\frac{|\nabla^{\rm rad} u|^2}{|x|^2}\diff x\ge\frac{n^2(n-4)^2}{16}\int_\Omega\frac{u^2}{|x|^2},\qquad\mbox{(if \(n\ge 5\))}.
    \]
    
    As can be seen from this example, our results provide an efficient framework to prove inequalities of type~\eqref{eq:intro:Rellich:u} combining inequalities of type~\eqref{eq:intro:radial} involving radial derivatives. They can also be used to prove inequalities of type~\eqref{eq:intro:Rellich:gru} but at the price of some additional condition, typically involving dimension constraints. When presenting applications, we stick to the first approach; investigation of additional byproducts of type~\eqref{eq:intro:Rellich:gru} and related positivity testing of \(E_2\) is left to the interested reader. 

    The paper is structured as follows. In Section~\ref{sec:prelim} we present some relevant preliminary definitions and results concerning space forms and Riccati pairs. We introduce the concept of dual Riccati pairs, which is a slight modification of the original concept, arising from the higher order convexity arguments. In Section~\ref{sec:gfi} we prove our general functional inequalities in Theorem~\ref{thm:main2} \&~\ref{thm:main2}. Finally, in Section~\ref{sec:application}, a number of applications are presented on both Euclidean and hyperbolic spaces, highlighting the advantages and limitations of our method.  Additionally, the connection between Riccati pairs and Bessel pairs is discussed, as well. 
    
    \section{Preliminaries}\label{sec:prelim}
    In this section we present those definitions and results that are necessary to present our main results. First, we recall the basic concepts related to the non-positively curved model space forms. Next, we present a simplified definition of Riccati pairs, specifically refactored for these manifolds. Finally, we introduce the notion of dual Riccati pairs and discuss its relation with the original concept.

    \subsection{Space forms} The non-positively curved \emph{model space form} \((M^n_\kappa,g_\kappa)\) is an \(n\)-dimensional Riemannian manifold, with constant sectional curvature \({\bf K}=-\kappa^2\), for some \(\kappa\ge0\). More precisely,
    \begin{itemize}
        \item if \(\kappa=0\), then \(M^n_\kappa\) is the Euclidean space \(\mathbb{R}^n\),
        \item if \(\kappa>0\), then  \(M^n_\kappa\) is the hyperbolic space \(\mathbb{H}_\kappa^n=\{x\in\mathbb{R}^n:|x|<1\}\),
    \end{itemize}
    endowed with the metric \((g_\kappa)_{ij}(x)=p_\kappa^2(x)\delta_{ij},\) where
    \[
        p_\kappa(x)=\begin{cases}
            1,&\mbox{if }\kappa=0,\\
            \frac{2}{\kappa(1-|x|^2)},&\mbox{if } \kappa>0,
        \end{cases}        
    \]
    and \(\delta_{ij}\) is the \emph{Kronecker symbol}. The canonical volume form, the Riemannian gradient and the Laplace-Beltrami operator can be written as
    \[
        \diff x_\kappa=p^n_\kappa(x)\diff x,\quad
        \nabla_\kappa u=\frac{\nabla u}{p_\kappa^2}\quad\mbox{and}\quad \Delta_\kappa u=p_\kappa^{-n}\di(p_\kappa^{n-2}\nabla u),
    \]
    where \(u\) is a smooth function, while \(\di\) and \(\nabla\) denote the Euclidean divergence and gradient, respectively. Hereafter, for pure Euclidean quantities we omit the subscript \(\kappa\). 

    We use the notation \(d_\kappa(x,y)\) for the Riemannian distance between \(x,y\in M^n_\kappa\). For a fixed \(x_0\in M^n_\kappa\) denote \(d_{x_0}(x)=d_\kappa(x_0,x)\) the distance from \(x_0\).  By the \emph{eikonal equation} one has
    \[
        |\nabla_\kappa d_{x_0}(x)|=1,\qquad \forall x\in \Omega\setminus\{x_0\}.
    \]

    The radial derivative of a smooth function \(u\) is defined by
    \[
        \nabla_\kappa^{\rm rad} u = \langle \nabla_\kappa u,  \nabla_\kappa d_{x_0}\rangle,    
    \]
    and it satisfies the following Cauchy-Schwarz inequality:
    \begin{equation}\label{eq:cauchy}
        |\nabla_\kappa^{\rm rad} u|^2=\langle \nabla_\kappa u,  \nabla_\kappa d_{x_0}\rangle^2\le |\nabla_\kappa u|^2\cdot |\nabla_\kappa d_{x_0}|^2= |\nabla_\kappa u|^2.     
    \end{equation}

    The laplacian of the distance function can be computed as 
    \begin{equation}\label{eq:lap:def}
        \Delta_\kappa d_{x_0}=(n-1){\bf ct}_\kappa(d_{x_0}),
    \end{equation}
    where \({\bf ct}_\kappa\colon(0,\infty)\to(0,\infty)\) is defined by
    \[
        {\bf ct}_\kappa(t)=\begin{cases}
            \frac{1}{t},&\text{if } \kappa=0,\\
            \kappa\coth(\kappa t),&\text{if } \kappa>0.
        \end{cases}
    \]

    For any vector field \(V\in TM^n_\kappa\) the Hessian of the distance function verifies the relation
    \begin{equation}\label{eq:hess:def}
        \he_\kappa(d_{x_0})(V,V) = {\bf ct}_\kappa(d_{x_0}) (|V|^2 - \langle V,\nabla_\kappa d_{x_0}\rangle^2).
    \end{equation}

    If \(u_1\) and \(u_2\) are smooth functions on \(\Omega\subseteq M\), then then following integration by parts formula holds:
    \[
        \int_{\Omega} u_1\Delta_\kappa u_2\diff x_\kappa = -\int_{\Omega} \langle\nabla_\kappa u_1,\nabla_\kappa u_2\rangle\diff x_\kappa.
    \]

    We refer to Gallot,  Hulin, and  Lafontaine~\cite{gallot2004riemmanian} for more details on space forms.

    \subsection{Riccati pairs} Let \(\kappa\ge0\) and define \(L_\kappa\colon(0,\infty)\to(0,\infty)\), \(L_\kappa(t)=(n-1){\bf ct}_\kappa(t).\) A simplified definition of Riccati pairs refactored for space forms is as follows; for the original version, see~\cite[Definition 3.1]{riccatipair2023}. 

    \begin{definition}\label{def:Riccati:pair}
        Let \(\kappa\ge0\) and \(\Omega\subset M^n_\kappa\) be open, \(x_0\in\Omega\), \(\rho\) denote the Riemannian distance from \(x_0\) and \(w,W\colon(0,\sup\rho)\to[0,\infty)\) be smooth functions. The couple \((L_\kappa,W)\) is a \emph{\((\rho,w)\)-Riccati pair} if there exist a smooth function  \(G\colon(0,\sup\rho)\to\mathbb{R}\) such that the following ordinary differential inequality holds:
        \begin{equation}\label{eq:Riccati:odi}
            G'(t)+\left(L_\kappa(t)+\frac{w'(t)}{w(t)}\right)G(t)-G(t)^2\ge W(t),\qquad \forall t\in(0,\sup\rho).
        \end{equation}
        A function \(G\) satisfying~\eqref{eq:Riccati:odi} is said to be \emph{\((\rho,w)\)-admissible} for the Riccati pair \((L_\kappa,W)\) on \(\Omega\).
    \end{definition}
    The above inequality is a driving force for the Hardy-type functional inequalities. However, as we shall see soon, the Rellich-type inequalities work well with its `dual' version, which can be stated as follows.
    \begin{definition}\label{def:Riccati:pair:dual}
        Let \(\kappa\ge0\), \(\Omega\subset M^n_\kappa\) be open, \(x_0\in\Omega\), \(\rho\) denote the Riemannian distance from \(x_0\) and \(v,V\colon(0,\sup\rho)\to[0,\infty)\) be smooth functions. The couple \((L_\kappa,V)\) is a \emph{\((\rho,v)\)-dual Riccati pair} if there exist a smooth function  \(H\colon(0,\sup\rho)\to\mathbb{R}\) such that the following ordinary differential inequality holds:
        \begin{equation}\label{eq:Riccati:odi:dual}
            -H'(t)+\left(L_\kappa(t)-\frac{v'(t)}{v(t)}\right)H(t)-H(t)^2\ge V(t),\qquad \forall t\in(0,\sup\rho).
        \end{equation}
        A function \(H\) satisfying~\eqref{eq:Riccati:odi:dual} is said to be \emph{\((\rho,v)\)-admissible} for the dual Riccati pair \((L_\kappa,V)\) on \(\Omega\).
    \end{definition}

    We conclude this section by showing that the two concepts are equivalent. Indeed, the following changes of functions provide the transition between the two equations:
    \[
        H(t) = L_\kappa(t)-G(t),\quad  v(t)=w(t)\quad\mbox{and}\quad V(t)=W(t)-\frac{v'(t)}{v(t)}L_\kappa(t)-L_\kappa'(t).
    \]

    \section{General functional inequalities}\label{sec:gfi}
    In this section we present our general functional inequalities based on the Riccati and dual Riccati pairs. Our first main result reads as follows:
    \begin{theorem}\label{thm:main1}
        Let \(\kappa\ge 0\), \(\Omega\in M^n_\kappa\) be open, \(x_0\in\Omega\), \(\rho\) denote the Riemannian distance from \(x_0\) and suppose that \(H\) is \((\rho,v)\)-admissible for the dual Riccati pair \((L_\kappa,V)\) on \(\Omega\). The following statements hold.
        \begin{enumerate}[label=\rm(\roman*)]
            \item\label{c:i} For every \(u\in C_0^\infty(\Omega)\) one has 
            \begin{equation}\label{eq:gen:1}
                \int_\Omega v(\rho)|\Delta_\kappa u|^2\diff x_\kappa\ge \int_\Omega v(\rho)V(\rho)|\nabla_\kappa^{\rm rad} u|^2\diff x_\kappa,
            \end{equation}
            provided that  \(E_1(t)=(v(t)H(t))'+v(t)H(t)(L_\kappa(t)-2{\bf ct}_\kappa(t))\ge 0\), for all \(t>0\);\medskip
            \item\label{c:ii} For every \(u\in C_0^\infty(\Omega)\) one has 
            \begin{equation}\label{eq:gen:2}
                \int_\Omega v(\rho)|\Delta_\kappa u|^2\diff x_\kappa\ge \int_\Omega v(\rho)V(\rho)|\nabla_\kappa u|^2\diff x_\kappa,
            \end{equation}
            provided that \(E_2(t)=2(v(t)H(t))'+v(t)H(t)(H(t)-2{\bf ct}_\kappa(t))\ge 0\), for all \(t>0\). 
        \end{enumerate}
    \end{theorem}
    \begin{proof}
        The convexity of $\xi\mapsto|\xi|^{2}$ implies 
        \begin{equation}\label{eq:convexity}
            |\xi|^{2}\geq 2\xi\eta-|\eta|^{2}
        \end{equation}
        for every \(\xi,\eta\in \mathbb{R}\). Choose 
        \[
            \xi=\Delta_\kappa u\quad\mbox{and}\quad \eta = H(\rho)\langle\nabla_\kappa u,\nabla_\kappa\rho\rangle 
        \] 
        to obtain 
        \[
            |\Delta_\kappa u|^2\ge 2H(\rho) \Delta_\kappa u  \langle\nabla_\kappa u,\nabla_\kappa\rho\rangle-H(\rho)^{2}\langle\nabla_\kappa u,\nabla_\kappa\rho\rangle ^2.    
        \]

        Multiplying both sides by \(v(\rho)\ge0\) and integrating over \(\Omega\) yields
        \[
            \int_\Omega v(\rho)|\Delta_\kappa u|^2\diff x_\kappa
            \ge 2\int_\Omega v(\rho) H(\rho) \Delta_\kappa u  \langle\nabla_\kappa u,\nabla_\kappa\rho\rangle\diff x_\kappa
            -\int_\Omega v(\rho) H(\rho)^{2}\langle\nabla_\kappa u,\nabla_\kappa\rho\rangle ^2\diff x_\kappa.
        \]

        Applying an integration by parts to the second term yields
        \begin{align*}
            I&=2\int_\Omega v(\rho)H(\rho) \Delta_\kappa u \langle\nabla_\kappa u,\nabla_\kappa\rho\rangle\diff x_\kappa
            =-2\int_\Omega \nabla_\kappa\left(v(\rho) H(\rho)\langle\nabla_\kappa u,\nabla_\kappa\rho\rangle\right)\nabla_\kappa u \diff x_\kappa\\
            &=-2\int_\Omega (v(\rho)H(\rho))'\langle \nabla_\kappa u,\nabla_\kappa \rho \rangle^2\diff x_\kappa-2\int_\Omega v(\rho)H(\rho)\he_\kappa(u)(\nabla_\kappa u,\nabla_\kappa\rho) \diff x_\kappa \\
            &\qquad-2\int_\Omega v(\rho)H(\rho)\he_\kappa(\rho)(\nabla_\kappa u, \nabla_\kappa u) \diff x_\kappa.
        \end{align*}

        First, an other integration by parts and expression~\eqref{eq:lap:def} yield
        \begin{align*}
            -2\int_\Omega v(\rho)H(\rho)\he_\kappa(u)(\nabla_\kappa u,\nabla_\kappa\rho) \diff x_\kappa&=-\int_\Omega v(\rho)H(\rho)\langle\nabla_\kappa|\nabla_\kappa u|^2,\nabla_\kappa \rho\rangle\diff x_\kappa\\
            &=\int_\Omega \left((v(\rho)H(\rho))'+v(\rho)H(\rho)L_\kappa(\rho)\right) |\nabla_\kappa u|^2 \diff x_\kappa.
        \end{align*}

        Next, relation ~\eqref{eq:hess:def} for \(V=\nabla_\kappa u\) implies
        \begin{align*}
            -2\int_\Omega v(\rho)H(\rho)\he_\kappa(\rho)(\nabla_\kappa u, \nabla_\kappa u) \diff x_\kappa&=2\int_\Omega v(\rho)H(\rho){\bf ct}_\kappa(\rho)\left(\langle\nabla_\kappa u,\nabla_\kappa\rho\rangle^2-|\nabla_\kappa u|^2\right)\diff x_\kappa.
        \end{align*}

        Finally, by the above computation we obtain that
        \begin{align}
            \int_\Omega v(\rho)|\Delta_\kappa u|^2\diff x_\kappa&\ge \int_\Omega \left(-2(v(\rho)H(\rho))'+2v(\rho)H(\rho){\bf ct}_\kappa(\rho)-v(\rho)H(\rho)^2\right)|\nabla_\kappa^{\rm rad} u|^2\diff x_\kappa \nonumber\\
            \label{eq:intermediate}&\qquad+\int_\Omega \left((v(\rho)H(\rho))'+v(\rho)H(\rho)L_\kappa(\rho)-2v(\rho)H(\rho){\bf ct}_\kappa(\rho) \right) |\nabla_\kappa u|^2 \diff x_\kappa.
        \end{align}
        
        If \(E_1\ge0\), then inequality~\eqref{eq:intermediate} and the Cauchy-Schwarz inequality
       ~\eqref{eq:cauchy} imply
        \[
            \int_\Omega v(\rho) |\Delta_\kappa u|^2\diff x_\kappa\ge \int_\Omega \left(-(v(\rho)H(\rho))'+v(\rho)H(\rho)L_\kappa(\rho) - v(\rho)H(\rho)^2\right) |\nabla_\kappa^{\rm rad} u|^2\diff x_\kappa,
        \]
        whence the dual Riccati ODI~\eqref{eq:Riccati:odi:dual} yields the desired inequality~\eqref{eq:gen:1}. 
        
        If \(E_2\ge0\), then again inequalities~\eqref{eq:intermediate},~\eqref{eq:cauchy} and~\eqref{eq:Riccati:odi:dual} imply~\eqref{eq:gen:2}, concluding the proof.
    \end{proof}

    Our second main result is tailored to complement Theorem~\ref{thm:main2}/\ref{c:i}. It can be stated as follows.
    \begin{theorem}\label{thm:main2}
        Let \(\kappa\ge 0\), \(\Omega\in M^n_\kappa\) be open, \(x_0\in\Omega\), \(\rho\) denote the Riemannian distance from \(x_0\) and suppose that \(G\) is \((\rho,w)\)-admissible for the Riccati pair \((L_\kappa,W)\) on \(\Omega\). Then for every \(u\in C_0^\infty(\Omega)\) one has 
            \begin{equation}\label{eq:genp:1}
                \int_\Omega w(\rho)|\nabla_\kappa^{\rm rad} u|^2\diff x_\kappa\ge \int_\Omega w(\rho)W(\rho)u^2\diff x_\kappa.
            \end{equation}
    \end{theorem}
    \begin{proof} Choose \(\xi=\langle \nabla_\kappa u,\nabla_\kappa \rho\rangle\) and \(\eta = -G(\rho)u\) in the convexity inequality~\eqref{eq:convexity} to obtain
        \[
            \langle \nabla_\kappa u,\nabla_\kappa \rho\rangle^2\ge-2G(\rho) u\langle \nabla_\kappa u,\nabla_\kappa \rho\rangle - G(\rho)^2u^2.
        \]
        Multiplying both sides by \(w(\rho)\ge 0\) and integrating over \(\Omega\) yields
        \[\int_\Omega w(\rho) \langle \nabla_\kappa u,\nabla_\kappa \rho\rangle^2\diff x_\kappa\ge-\int_\Omega 2w(\rho)G(\rho) u\langle \nabla_\kappa u,\nabla_\kappa \rho\rangle\diff x_\kappa - \int_\Omega w(\rho)G(\rho)^2u^2\diff x_\kappa.\]
        An integration by parts implies
        \[-\int_\Omega 2w(\rho)G(\rho) u\langle \nabla_\kappa u,\nabla_\kappa \rho\rangle\diff x_\kappa=\int_\Omega (w(\rho)G'(\rho)+w'(\rho)G(\rho)+w(\rho)G(\rho)L_\kappa(\rho))u^2\diff x_\kappa,\]
        whence the Riccati ODI~\eqref{eq:Riccati:odi} finishes the proof of the desired inequality~\eqref{eq:genp:1}.
    \end{proof} 
    \begin{remark} We note that the proof of Theorem~\ref{thm:main2} is analogous to~\cite[Theorem 3.2]{riccatipair2023}, where similar steps with choices \(\xi = \nabla_\kappa u\) and \(\eta = -G(\rho) u \nabla_\kappa \rho\) imply the following inequality: \[\int_\Omega w(\rho)|\nabla_\kappa u|^2\diff x_\kappa\ge \int_\Omega w(\rho)W(\rho)u^2\diff x_\kappa,\qquad \forall u\in C_0^\infty(\Omega).\]
    \end{remark}

    \section{Applications}\label{sec:application}
    In this section, we provide alternative proofs for some Rellich type inequalities using the combination of Theorem~\ref{thm:main1}/\ref{c:i} and Theorem~\ref{thm:main2}. The examples are chosen to demonstrate both the advantages and limitations of our method. The Euclidean and hyperbolic cases are discussed separately; in the former case the relation between Riccati and Bessel pairs is discussed, as well.

    \subsection{Rellich inequalities on Euclidean spaces}\label{ssec:eu} In the Euclidean setting, a number of well-known Rellich inequalities are discussed by Ghoussoub and Moradifam~\cite{ghoussoub2011bessel} in terms of Bessel potentials or Bessel pairs. After summarizing the related concepts, we discuss two novel inequalities among them, which highlight the capabilities of our method. For the original proof of the selected inequalities \eqref{ineq:ex1} \& \eqref{ineq:ex2} see Adimurthi,  Grossi, and Santra~\cite{adimurthi2006optimal}.

    A function \(Z>0\) is said to be a \emph{Bessel potential} on \((0,R)\) if there exist a constant \(c>0\) and a function  \(z>0\) such that following ordinary differential equation holds: 
    \begin{equation}\label{eq:Bessel:potential}
        z''(t)+\frac{z'(t)}{t}+c\cdot Z(t)\cdot z(t)=0,\qquad\forall t\in(0,R). 
    \end{equation}

    The couple \((X,Y)\) of positive functions is said to be a \emph{Bessel pair} if there exist a constant \(C>0\) and a function  \(y>0\) such that following ordinary differential equation holds: 
    \begin{equation}\label{eq:Bessel:pair}
        y''(t)+\left(\frac{n-1}{t}+\frac{X'(t)}{X(t)}\right)y'(t)+\frac{C\cdot Y(t)}{X(t)}y(t)=0,\qquad\forall t\in(0,R). 
    \end{equation}
    
    The next Proposition should clarify the relation between several notions.
    \begin{proposition}\label{prop:eu:relations} 
        The following statements hold.
        \begin{enumerate}[label=\rm(\roman*)]
            \item\label{it:rel:i} If equation~\eqref{eq:Bessel:potential} holds on \((0,R)\) then equation~\eqref{eq:Bessel:pair} holds as well,  for the choices:
            \[
                y(t)=z(t)\cdot t^\frac{2-n}{2},\quad C=1,\quad X(t)\equiv 1\quad\mbox{and}\quad Y(t)=\frac{(n-2)^2}{4t^2}+c\cdot Z(t).
            \]
            \item\label{it:rel:ii} If equation~\eqref{eq:Bessel:potential} holds on \((0,R)\) then relation~\eqref{eq:Riccati:odi} holds as well with equality, for the choices:
            \[G(t)=-\frac{z'(t)}{z(t)}+\frac{n-2}{2t},\quad L_\kappa(t)=\frac{n-1}{t},\quad w(t)\equiv 1\quad\mbox{and}\quad W(t)=\frac{(n-2)^2}{4t^2}+c\cdot Z(t).\] 
            \item\label{it:rel:iii} If equation~\eqref{eq:Bessel:potential} holds on \((0,R)\) then relation~\eqref{eq:Riccati:odi:dual} holds as well with equality, for the choices:
            \[H(t)=\frac{n}{2t}+\frac{z'(t)}{z(t)},\quad L_\kappa(t)=\frac{n-1}{t},\quad  v(t)\equiv1 \quad\mbox{and}\quad V(t)=\frac{n^2}{4t^2}+c\cdot Z(t).\]
            \item\label{it:rel:iv} If equation~\eqref{eq:Bessel:pair} holds on \((0,R)\) then relation~\eqref{eq:Riccati:odi} holds as well with equality, for the choices:
            \[
                G(t)=-\frac{y'(t)}{y(t)},\quad L_\kappa(t)=\frac{n-1}{t},\quad w(t)=X(t)\quad\mbox{and}\quad W(t)=\frac{C\cdot Y(t)}{X(t)}.
            \] 
        \end{enumerate}
    \end{proposition}

    In the sequel we use the following result, which is a simple consequence of~\cite[Theorem 2.6]{ghoussoub2011bessel}. Its proof is based on the comparison of solutions of~\eqref{eq:Bessel:potential}.
    \begin{proposition}\label{prop:eu:aux}
        Suppose that \(Z\) is a Bessel potential on \((0,R)\) with best constant \(c\), such that
        \begin{equation}\label{eq:prop:aux:condition}
            \frac{Z'(t)}{Z(t)}=-\frac{\lambda}{t}+f(t),\quad\mbox{where } \lambda<n-2,\, f(t)\ge0\mbox{ and }\lim_{t\to0}tf(t)=0.
        \end{equation}
        Then the couples 
        \(\left(\frac{1}{t^2},\frac{(n-4)^2}{4t^4}+\frac{c\cdot Z}{t^2}\right)\) and \(\left(Z,\frac{(n-\lambda-2)^2\cdot Z}{4t^2}\right)\)
        are Bessel pairs on \((0,R)\), with \(C=1\).
    \end{proposition}

    Using our general inequalities we can state the following.
    \begin{theorem}\label{thm:eu}
        Let \(n\ge 5\) and \(B\subseteq\mathbb{R}^n\) be a ball centered at the origin with radius \(R>0\). Suppose that \(Z\) is a Bessel potential on \((0,R)\) with solution \(z\) and best constant \(c\), such that condition~\eqref{eq:prop:aux:condition} holds. Define \(H(t)=\frac{n}{2t}+\frac{z'(t)}{z(t)}\). If \(E_1(t)=H'(t)+H(t)\cdot \frac{n-3}{t}\ge0\) for all \(t\in (0,R)\), then for every \(u\in C_0^\infty(B)\) one has
        \[
            \int_B |\Delta u|^2\diff x \ge \frac{n^2(n-4)^2}{16}\int_B \frac{u^2}{|x|^4}\diff x +c\left(\frac{n^2}{4}+\frac{(n-\lambda-2)^2}{4}\right)\int_B\frac{Z(|x|)u^2}{|x|^2}\diff x.
        \]
    \end{theorem}
    \begin{proof} 
        Proposition~\ref{prop:eu:relations}/\ref{it:rel:iii} implies that the dual Riccati ODI~\eqref{eq:Riccati:odi:dual} is verified for 
        \[
            H(t)=\frac{n}{2t}+\frac{z'(t)}{z(t)},\quad L_\kappa(t)=\frac{n-1}{t},\quad  v(t)\equiv1 \quad\mbox{and}\quad V(t)=\frac{n^2}{4t^2}+c\cdot Z(t).
        \]
        On the one hand, Theorem~\ref{thm:main1}/\ref{c:i} yields 
        \[
            \int_B |\Delta u|^2\diff x\ge \frac{n^2}{4}\int_B \frac{|\nabla^{\rm rad}u|^2}{|x|^2}\diff x +c\int_B Z(t)|\nabla^{\rm rad}u|^2\diff x.
        \]
        On the other hand, using Proposition~\ref{prop:eu:aux} \& \ref{prop:eu:relations}/\ref{it:rel:iv} and Theorem~\ref{thm:main2} we obtain the following inequalities:
        \begin{align*}
            \int_B \frac{|\nabla^{\rm rad}u|^2}{|x|^2}\diff x&\ge\frac{(n-4)^2}{4}\int_B\frac{u^2}{|x|^4}\diff x+c\int_B\frac{Z(|x|)u^2}{|x|^2}\diff x,\\
            \int_B Z(|x|)|\nabla^{\rm rad}u|^2\diff x&\ge \frac{(n-\lambda-2)^2}{4}\int_B\frac{Z(|x|)u^2}{|x|^2}\diff x.
        \end{align*}
        Combining these results yield the desired inequality.
    \end{proof}

    We note that the proof of this result is analogous to \cite[Theorem 3.6]{ghoussoub2011bessel}. A Corollary of Theorem \ref{thm:eu} can be stated as follows. Hereafter, for any function \(h\) let us denote 
    \[h_{[0]}(t)=t,\quad h_{[1]}(t)=h(t)\quad\mbox{and}\quad h_{[i]}(t)=h(h_{[i-1]}(t)),\qquad \forall i \ge 2.\]
    \begin{corollary}
        Let \(n\ge 5\) and \(B\subseteq\mathbb{R}^n\) be a ball centered at the origin with radius \(R>0\). If \(k\ge 1\) and \(r=R\cdot \exp_{[k-1]}(e)\), then for every \(u\in C_0^\infty(\Omega)\) one has 
        \begin{equation}\label{ineq:ex1}
            \int_B |\Delta u|^2\diff x\ge \frac{n^2(n-4)^2}{16}\int_B \frac{u^2}{|x|^4}\diff x + \left(1+\frac{n(n-4)}{8}\right)\sum_{j=1}^k\int_B \frac{u^2}{|x|^4}\left(\prod_{i=1}^j\log_{[i]}\left(\frac{r}{|x|}\right)\right)^{-2}\diff x.
        \end{equation}
    \end{corollary}
    \begin{proof}
        It is straightforward to show that 
        \[
            Z_{k,r}(t)=\sum_{j=1}^k\frac{1}{t^2}\left(\prod_{i=1}^j\log_{[i]}\left(\frac{r}{t}\right)\right)^{-2}  \quad\mbox{and}\quad 
            z_{k,r}(t)=\left(\prod_{i=1}^k\log_{[i]}\left(\frac{r}{t}\right)\right)^\frac{1}{2}
        \]
        satisfy~\eqref{eq:Bessel:potential} for every \(t\in(0,R)\), with the best constant \(c=\frac{1}{4}\); moreover condition~\eqref{eq:prop:aux:condition} holds for \(\lambda=2\). Once we show the positivity of \(E_1(t)=H'(t)+H(t)\cdot \frac{n-3}{t}\) on \((0,R)\), where \(H(t)=\frac{n}{2t}+\frac{z_{k,r}'(t)}{z_{k,r}(t)}\), we can apply Theorem~\ref{thm:eu} which provides the desired inequality.

        To prove the positivity of \(E_1\), by relation~\eqref{eq:Bessel:potential} it is enough to show that
        \[
            -q_{k,r}(t)^2+(n-4)q_{k,r}(t)+\frac{n(n-4)}{2}-\frac{t^2}{4}Z_{k,r}(t)\ge0,\qquad  \forall t\in(0,R),
        \]
        where \(q_{k,r}(t)=\frac{t\cdot z_{k,r}'(t)}{z_{k,r}(t)}\). To see this, observe that 
        \[
            \log_{[k]}\left(\frac{r}{t}\right)\ge 1\quad\mbox{and}\quad
            \log_{[i]}\left(\frac{r}{t}\right)\ge\exp_{[k-1-i]}(e)\ge e\ge 2,\qquad \forall t\in(0,R),\ i\in\{1,2,\dots, k-1\},
        \]
        moreover, for the products, the following estimates hold:
        \[
            \prod_{i=1}^k\log_{[i]}\left(\frac{r}{t}\right)\ge 2^{k-1}\quad\mbox{and}\quad  \prod_{i=1}^j\log_{[i]}\left(\frac{r}{t}\right)\ge 2^j,\qquad \forall t\in(0,R),\ j\in\{1,2,\dots, k-1\}.
        \]

        On the one hand, by simple computation we obtain for all \(t\in(0,R)\) that  \(q_{k,r}(t)< 0\) and 
        \[
            q_{k,r}(t)=-\frac{1}{2}\sum_{j=1}^k\left(\prod_{i=1}^{j}\log_{[i]}\left(\frac{r}{t}\right)\right)^{-1}\ge -\frac{1}{2}\left(1+\frac{1}{2}+\frac{1}{2^2}\dots +\frac{1}{2^{k-1}}+\frac{1}{2^{k-1}}\right)> -1.
        \]

        On the other hand, for all \(t\in(0,R)\) one has
        \[
            -\frac{t^2Z_{k,r}(t)}{4}\ge -\frac{1}{4}\left(1+\frac{1}{2^2}+\frac{1}{2^4}+\dots+\frac{1}{2^{2(k-1)}}+\frac{1}{2^{2(k-1)}}\right)\ge -\frac{1}{2}.
        \]
        
        By the above computations it is enough to show that
        \[
            -q^2+(n-4)q+\frac{n^2-4n-1}{2}\ge 0,\qquad \forall q\in(-1,0).
        \]
        The roots of the above polynomial are \(q_\pm=\frac{n-4}{2}\pm\frac{1}{2}\sqrt{3n^2-16n+14}\), thus a basic computation implies \(q_-\le-1<0<q_+\) for every \(n\ge 5\), which concludes the the proof. 
    \end{proof}

    The limitations of Theorem~\ref{thm:eu} are illustrated by the following example.
    \begin{example} 
        Define \(\ell(t)=\frac{1}{1-\log(t)}\), let \(k\ge 1\) and \(R>0\) and consider the inequality 
        \begin{equation}\label{ineq:ex2}
            \int_\Omega |\Delta u|^2\diff x\ge \frac{n^2(n-4)^2}{16}\int_\Omega \frac{u^2}{|x|^4}\diff x + \left(1+\frac{n(n-4)}{8}\right)\sum_{j=1}^k\int_\Omega \frac{u^2}{|x|^4}\prod_{i=1}^j\ell_{[i]}^2\left(\frac{|x|}{R}\right)\diff x.
        \end{equation}
        This inequality is generated by the Bessel potential with parameters 
        \[
            \widetilde{Z}_{k,R}(t)=\sum_{j=1}^k\frac{1}{t^2}\prod_{i=1}^j\ell_{[i]}^2\left(\frac{t}{R}\right),\quad
            \widetilde{z}_{k,R}=\left(\prod_{i=1}^k\ell_{[i]}\left(\frac{t}{R}\right)\right)^{-\frac{1}{2}}\quad\mbox{and}\quad c=\frac{1}{4}, \qquad \forall t\in(0,R),
        \] 
        moreover, condition \eqref{eq:prop:aux:condition} holds for \(\lambda=2\). Let \(H(t)=\frac{n}{2t}+\frac{\widetilde{z}_{k,r}'(t)}{\widetilde{z}_{k,r}(t)}\) and observe that 
        \[\lim_{t\to R} E_1(t)=\frac{n-k}{2},\] 
        thus for \(k>n\) the positivity condition does not hold, hence Theorem~\ref{thm:eu} can not be applied.
    \end{example}
    
    \subsection{Rellich inequalities on hyperbolic spaces} Our results can be used to provide Rellich inequalities on hyperbolic spaces, as well. To illustrate this, we provide an improved Rellich inequality, which follows from Theorem~\ref{thm:main2}/\ref{c:i} \& \ref{thm:main2} for specific choices of parameters functions. We also note that, other alternative choices lead us to several different inequalities; there are plenty of possibilities.

    The first result in this setting, an interpolation inequality, which can be stated as follows.
    \begin{theorem} 
        Let \(\kappa>0\), \(n\ge 5\) and \(\Omega\subseteq \mathbb{H}^n_\kappa\) be an open domain. Fix \(x_0\in\Omega\) and denote \(\rho=d_{x_0}\) the Riemannian distance from \(x_0\). Then for every \(u\in C_0^\infty(\Omega)\) one has 
        \begin{align}
            \int_\Omega |\Delta_\kappa u|^2\diff x_\kappa
            &\ge\kappa^2\lambda\int_\Omega |\nabla^{\rm rad}_\kappa u|^2\diff x_\kappa
            +h_n^2(\lambda)\int_\Omega \frac{|\nabla^{\rm rad}_\kappa u|^2}{\rho^2}\diff x_\kappa
            +\kappa^2\left(\frac{n^2}{4}-h_n^2(\lambda)\right)\int_\Omega \frac{|\nabla^{\rm rad}_\kappa u|^2}{\sinh^2(\kappa\rho)}\diff x_\kappa\nonumber\\
            &\qquad+\gamma_n(\lambda)h_n(\lambda)\int_\Omega \frac{(\rho\,{\bf ct}_\kappa(\rho)-1)}{\rho^2}|\nabla^{\rm rad}_\kappa u|^2\diff x_\kappa,\label{ineq:hyp:gen:r}
        \end{align}
        where \(0\le\lambda\le\frac{(n-1)^2}{4}\), \(\gamma_n(\lambda)=\sqrt{(n-1)^2-4\lambda}\) and \(h_n(\lambda)=\frac{\gamma_n(\lambda)+1}{2}\).
    \end{theorem}
    \begin{proof} 
        In Theorem~\ref{thm:main1} let us choose \(v(t)\equiv1\), as well as,
        \begin{align*}
            H(t)&=\left(\frac{n}{2}-h_n(\lambda)\right){\bf ct}_\kappa(t)+\frac{h_n(\lambda)}{t},\qquad\forall t>0,\\
            V(t)&=\kappa^2\lambda+\frac{h_n^2(\lambda)}{t^2}+\kappa^2\left(\frac{n^2}{4}-h_n^2(\lambda)\right)\frac{1}{\sinh^2(\kappa t)}+\gamma_n(\lambda)h_n(\lambda)\frac{t\,{\bf ct}_\kappa(t)-1}{t^2},\qquad\forall t>0.
        \end{align*}
        For these choices, the dual Riccati ODI~\eqref{def:Riccati:pair:dual} is verified with equality. Moreover 
        \[
            E_1(t)=\kappa^2\left(\frac{n}{2}-h_n(\lambda)\right)+\frac{h_n(\lambda)((n-3)t\,{\bf ct}_\kappa(t)-1)}{t^2}+\kappa^2(n-4)\left(\frac{n}{2}-h_n(\lambda)\right)\coth^2(\kappa t)
        \]
        is positive for all \(t>0\), since \(n\ge5\), \(\frac{1}{2}\le h_n(\lambda)\le\frac{n}{2}\) and \(t\,{\bf ct}_\kappa(t)>1\), \(\forall t>0\).
    \end{proof}
    \begin{corollary}\label{cor:hyp:1}
        Choose \(\lambda=0\) in~\eqref{ineq:hyp:gen:r} to obtain for every \(u\in C_0^\infty(\Omega)\) that
        \[
            \int_\Omega |\Delta_\kappa u|^2\diff x_\kappa
            \ge\frac{n^2}{4}\int_\Omega \frac{|\nabla^{\rm rad}_\kappa u|^2}{\rho^2}\diff x_\kappa
            +\frac{n(n-1)}{2}\int_\Omega \frac{(\rho\,{\bf ct}_\kappa(\rho)-1)}{\rho^2}|\nabla^{\rm rad}_\kappa u|^2\diff x_\kappa.
        \]
    \end{corollary}
    \begin{corollary}\label{cor:hyp:2}
        Choose \(\lambda=\frac{(n-1)^2}{4}\) in~\eqref{ineq:hyp:gen:r} to obtain for every \(u\in C_0^\infty(\Omega)\) that
        \[
            \int_\Omega |\Delta_\kappa u|^2\diff x_\kappa
            \ge\frac{(n-1)^2\kappa^2}{4}\int_\Omega |\nabla^{\rm rad}_\kappa u|^2\diff x_\kappa
            +\frac{1}{4}\int_\Omega \frac{|\nabla^{\rm rad}_\kappa u|^2}{\rho^2}\diff x_\kappa
            +\frac{(n^2-1)\kappa^2}{4}\int_\Omega \frac{|\nabla^{\rm rad}_\kappa u|^2}{\sinh^2(\kappa\rho)}\diff x_\kappa.
        \]
    \end{corollary}
    Lower order counterparts of Corollary \ref{cor:hyp:2} can be stated as follows.
    \begin{theorem}\label{thm:hyp:lower}
        Let \(\kappa>0\), \(n\ge 5\) and \(\Omega\subseteq \mathbb{H}^n_\kappa\) be an open domain. Fix \(x_0\in\Omega\) and denote \(\rho=d_{x_0}\) the Riemannian distance from \(x_0\). Then for every \(u\in C_0^\infty(\Omega)\) one has 
        \begin{align*}
            \int_\Omega |\nabla^{\rm rad}_\kappa u|^2\diff x_\kappa &\ge
            \frac{(n-1)^2\kappa^2}{4}\int_\Omega u^2\diff x_\kappa
            +\frac{1}{4}\int_\Omega\frac{u^2}{\rho^2}\diff x_\kappa
            +\frac{(n-1)(n-3)\kappa^2}{4}\int_\Omega\frac{u^2}{\sinh^2(\kappa\rho)}\diff x_\kappa,\\
            \int_\Omega \frac{|\nabla^{\rm rad}_\kappa u|^2}{\rho^2}\diff x_\kappa &\ge \frac{9}{4}\int_\Omega \frac{u^2}{\rho^4}\diff x_\kappa-(n-1)\int_\Omega \frac{{\bf ct}_\kappa(\rho)}{\rho^3}\diff x_\kappa+\frac{(n-1)^2\kappa^2}{4}\int_\Omega  \frac{u^2}{\rho^2}\diff x_\kappa\nonumber\\&\qquad+\frac{(n-1)(n-3)\kappa^2}{4}\int_\Omega \frac{u^2}{\rho^2\sinh^2(\kappa\rho)}\diff x,\\
            \int_\Omega \frac{|\nabla^{\rm rad}_\kappa u|^2}{\sinh^2(\kappa\rho)}\diff x_\kappa&\ge\frac{1}{4}\int_\Omega\frac{u^2}{t^2\sinh^2(\kappa\rho)}\diff x_\kappa+\frac{(n-3)^2\kappa^2}{4}\int_\Omega\frac{u^2}{\sinh^2(\kappa\rho)}\diff x_\kappa\nonumber\\
            &\qquad+\frac{(n-3)(n-5)\kappa^2}{4}\int_\Omega\frac{u^2}{\sinh^4(\kappa\rho)}\diff x_\kappa.
        \end{align*}
    \end{theorem}
    \begin{proof}
        All these inequalities follow from Theorem~\ref{thm:main2}. In each case, the choice of \(G\) is 
        \[
            G_1(t)=\frac{n-1}{2}{\bf ct}_\kappa(t)-\frac{1}{2t},\quad G_2(t)=\frac{n-1}{2}{\bf ct}_\kappa(t)-\frac{3}{2t}\quad\mbox{and}\quad G_3(t)=\frac{n-3}{2}{\bf ct}_\kappa(t)-\frac{1}{2t},
        \] respectively, while \(L_\kappa(t)=(n-1){\bf ct}_\kappa(t)\), and the functions \(w\) and \(W\) are needed to be chosen corresponding to the left and the right hand side of the given inequality. 
    \end{proof}
    Combining Theorem~\ref{thm:hyp:lower} and Corollary~\ref{cor:hyp:2} we obtain the following result, which concludes the paper.
    \begin{theorem}
        Let \(\kappa>0\), \(n\ge 5\) and \(\Omega\subseteq \mathbb{H}^n_\kappa\) be an open domain. Fix \(x_0\in\Omega\) and denote \(\rho=d_{x_0}\) the Riemannian distance from \(x_0\). Then for every \(u\in C_0^\infty(\Omega)\) one has 
        \begin{align*}
            \int_\Omega |\Delta_\kappa u|^2\diff x_\kappa
            &\ge \frac{(n-1)^4\kappa^2}{16}\int_\Omega u^2\diff x_\kappa
            +\frac{(n-1)^2\kappa^2}{8}\int_\Omega \frac{u^2}{\rho^2}\diff x_\kappa
            +\frac{(n-1)^2\kappa^2}{8}\int_\Omega \frac{u^2}{\rho^2\sinh^2(\kappa \rho)}\diff x_\kappa\\
            &\qquad+
            \frac{(n-1)(n-3)(n^2-2n-1)\kappa^4}{8}\int_\Omega\frac{u^2}{\sinh^2(\kappa\rho)}\diff x_\kappa-\frac{(n-1)\kappa}{4}\int_\Omega\frac{{\bf ct}_\kappa(\rho)|u^2|}{t^3}\diff x_\kappa\\
            &\qquad+\frac{(n^2-1)(n-3)(n-5)\kappa^4}{16}\int_\Omega\frac{u^2}{\sinh^4(\kappa\rho)}\diff x_\kappa+\frac{9}{16}\int_\Omega\frac{u^2}{\rho^4}\diff x_\kappa.
        \end{align*}    
    \end{theorem}
    \vspace*{0.15in}\noindent
    {\bf Acknowledgement.} The author would like to thank Professors Alexandru Kristály and Csaba Farkas for their support throughout the research process.

    \bibliographystyle{plain}
    \bibliography{references}

\end{document}